\documentclass[preprint,12pt]{imsart}

\newcommand{\vekk}[1]{}

\usepackage{amsthm,amsmath}
\usepackage{amsfonts}
\usepackage{amssymb}

\usepackage{tikz-cd}

\usepackage[a4paper,includeheadfoot,margin=2.54cm]{geometry}

\startlocaldefs

\vekk{

\def\@begintheorem#1#2{\it \theoremhook \trivlist \item[\hskip \labelsep{\bf #1\ #2}]}
\def\@opargbegintheorem#1#2#3{\it \theoremhook \trivlist
      \item[\hskip \labelsep{\bf #1\ #2\ (#3)}]}
\def\theoremhook{}

\parindent=0cm
\parskip=1ex
}

\newtheorem{theo}{Theorem}[section]

\newtheorem{lemma}{Lemma}[section]
\newtheorem{defin}{Definition}[section]

\vekk{
\newenvironment{proof}{{\it Proof.}}%
        {\ifvmode\else\unskip\fi ~\penalty10000 \hfill%
        $\Box$\vspace{0ex}}
}

\newcommand{\id}{\mathop{id}}

\newcommand{\st}{\mid}

\newcommand{\abs}[1]{ {\left | #1 \right |} }                   

\newcommand{\dlik}{\mathop{:=}\nolimits} 

\newcommand{\into}{\mbox{$\: \rightarrow \:$}}

\newcommand{\norm}[1]{\left \| #1\right\|}

\newcommand{\supp}{\mathop{\mbox{supp}}}

\newcommand{\imply}{\Rightarrow}

\newcommand{\RealN}{{\mbox{$\Bbb R$}}} 

\vekk{
\newcommand{\abs}[1]{ {\left | #1 \right |} }

\newcommand{\dlik}{\mathop{:=}\nolimits}

\newcommand{\id}{\mathop{id}\nolimits}
\newcommand{\imply}{\;\; \Rightarrow \;\;}
\newcommand{\into}{\mbox{$\: \rightarrow \:$}}
\newcommand{\norm}[1]{\left \| #1\right\|}

\newcommand{\RealN}{{\mbox{$\Bbb R$}}} 
\newcommand{\st}{\mid}
\newcommand{\supp}{\mathop{supp}\nolimits}
}
\newcommand{\Spec}{\mathop{Sp}\nolimits}

\theoremstyle{plain}
\newtheorem{proposition}{Proposition}
\theoremstyle{remark}
\usepackage{latexsym}
\usepackage[all]{xy} 
\CompileMatrices
\endlocaldefs

\begin{document}

\title{ 
Image transformations\\ on locally compact spaces
}

\runtitle{Image transformations}

\begin{aug}
\author{\fnms{Gunnar} \snm{Taraldsen}\corref{}\ead[label=e1]{Gunnar.Taraldsen@ntnu.no}}

\address{Trondheim, Norway.
\printead{e1}}

\runauthor{Taraldsen}

\affiliation{Norwegian University of Science and Technology}

\end{aug}

\begin{abstract}
An image is here defined to be a set which is either open or closed in $X$ and an image transformation $q$
is structure preserving in the following sense:
It corresponds to an algebra homorphism $q: A(a) \into A(q(a))$ for each
singly generated algebra $A (a) = \{\phi(a) \st \phi \in C(\RealN,\RealN), a \in C(X,\RealN)\}$.
We extend parts of J.F. Aarnes' results on quasi-measures, -states, -homomorphisms,
and image-transformations from the setting compact Hausdorff spaces
to locally compact Hausdorff spaces.
\end{abstract}

\maketitle

\vspace{0ex}


\tableofcontents

\section{Introduction and Definitions.}

In the $C^*$-algebraic formulation of quantum mechanics the bounded real
observables are identified with selfadjoint elements in
a $C^*$-algebra $A$.
A physical state can be considered to be
\cite[p.602]{AARNES:Q1}, 
\cite{SEGAL} an assignment of a probability
measure $\mu_a$ on the spectrum of each selfadjoint $a$.
The measure $\mu_a$ is interpreted as the probability
distribution which models the outcome in an experiment where 
the observable $a$ is measured.
This interpretation forces the consistency condition
$\mu_{\phi (a)} = \mu_a \circ \phi^{-1}$,
since a measurement of $a$ is also a measurement of any
observable $\phi (a)$ being a function of $a$.
The function $\phi$ is assumed to be continuous since 
the observable $\phi (a)$ is supposed to be an element in
the $C^*$-algebra.
The end result is that a physical state may be
identified with the functional $\mu$
given by $\mu (a) \dlik \mu_a (\id)$,
which fullfils the fundamental equation
\begin{equation}
\mu (\phi (a)) = \mu_a (\phi) \dlik \int \phi (t) \; \mu_a (dt) .
\end{equation}
Let $A (a) = \{\phi (a) \st \phi \in C (\RealN, \RealN) \}$ 
be the unital norm-closed real algebra of selfadjoint
elements generated by $a$.
A functional defined on the selfadjoint elements of 
a $C^*$-algebra which is linear on each
$A (a)$ is said to be quasi-linear.
It follows from the fundamental equation that $\mu$
is a quasi-linear functional.

In 1991 \cite{AARNES:QUASI} Aarnes presented the
first example of a proper quasi-linear functional.
The idea is to extend the Riesz representation theorem
from integrals to quasi-integrals to obtain a correspondence between 
quasi-integrals and quasi-measures.
Let ${\cal A} = {\cal A} (X)$ be the class of sets in a Hausdorff space $X$
which are either open or closed. 
A (compact-regular, additive, normalized) quasi-measure $\mu$
is a real valued function defined on
$\cal A$ with properties
(i) $\mu (X) = 1$,
(ii) $\mu$ is additive (on disjoints); 
$\mu (A \uplus B) = \mu (A) + \mu (B)$, and
(iii) $\mu$ is compact-regular: 
The measure of an open set $U$ equals the supremum
of the measures of compact sets $K$ contained in $U$.
A quasi-measure  is said to be simple if it only takes
the values $0$ and $1$.
Integration with respect to a quasi-measure $\mu$ is
defined as in \cite[p.46]{AARNES:QUASI}:
If $a: X \into \RealN$ is continuous, then
$\mu \circ a^{-1}$ is the restriction of a measure
$\mu_a$ on $\RealN$.
This gives a consistent family of measures and
$\rho (a) \dlik \mu (a) \dlik \int t \; \mu_a (dt)$ is a quasi-integral:
(i) $\rho: C_b (X) \into \RealN$;
(ii) $\rho : A (a) \into \RealN$ is linear;
(iii) $a \ge 0$ gives $\rho (a) \ge 0$;
(iv) $\rho (1) = 1$;
(v) $\rho (a) = \sup_{k \le a} \rho (k)$.
$\rho$ is simple if it is
multiplicative on each $A (a)$.

A random variable $X$ is a measurable
function from a probability space $\Omega$ to a measurable space $E$.
The set function $X^{-1}$ pulls measurable sets in $E$ back to
measurable sets in $\Omega$ in such a way that 
the measure $P$ on $\Omega$ is pushed to the measure
$P_X \dlik P \circ X^{-1}$ on $E$.
The fundamental change of variable formula
\begin{equation}
E (\phi (X)) \dlik \int \phi (X (\omega))\; P (d\omega)
= \int \phi (x)\; P_X (dx)
\end{equation}
shows that the expectation value of any (measurable) function
of $X$ may be computed from the distribution $P_X$ of $X$.
One may replace the set function $X^{-1}$ with a set function
$\psi$ with properties: 
(i) $\psi (E) = \Omega$,
(ii) $\psi (A^c) = \psi (A)^c$,
(iii) $\psi (A_1 \cup A_2 \cup \cdots)$ 
$= \psi (A_1) \cup \psi (A_2) \cup \cdots$,
and the measure $P$ is again pushed to a measure 
$\psi^* P \dlik P \circ \psi$ on $E$ with a integration result as
above.
 
These results can be generalized to the setting of quasi-measures.
The family of measurable sets is replaced by the family $\cal A$ of
images.
An image is a set which is open or closed.
The measurable functions are replaced by 
continuous functions,
and the result $\mu (\phi (a)) = \mu_a (\phi)$ is the generalization
of the change of variable formula.
The set functions $\psi$ are replaced by
image-transformations $q$.
An image-transformation $q$ from $X$ to $Y$ takes an
image $A$ in $X$ to an image $q (A)$ in $Y$, and has properties
(i) $q (X) = Y$,
(ii) $q (U)$ is open when $U$ is open,
(iii) $q$ is additive; $q (A \uplus B) = q (A) \uplus q (B)$, and
(iv) $q$ is compact-regular: Given an open set $U$ and a compact set
$K \subset q (U)$, there is a compact set $L \subset U$ such that
$K \subset q (L)$.
A quasi-measure $\mu$ on $Y$ is pulled to a quasi-measure
$q^* \mu \dlik \mu \circ q$ on $X$.
The integral $q (a)$ of a continuous bounded function $a$
on $X$ is the continuous bounded function $q (a)$ on $Y$ given by
$q (a) (y) \dlik (q^* \delta_y)(a)$.
The integral with respect to an image-transformation
is a (compact-regular) quasi-homomorphism from $C_b (X)$ to $C_b (Y)$:
(i) $q: A (a) \into A (q(a))$ is an algebra homomorphism
for each $a$, and
(ii) $q (a) (y) = \sup_k q (k) (y)$, where $k \le a$ has compact
support.
The integral above gives 1-1 correspondence between 
image-transformations and quasi-homomorphisms
for locally compact normal spaces.
The generalization of the change of variable formula is
\begin{equation}
(q^* \mu) (a) = \mu (q (a)) .
\end{equation}
The main result in this work is the characterization of all
image-transformations in terms of the 
Aarnes factorization theorem \cite{AARNES:IMAGE}:  
To any image-transformation $q$ from $X$ to $Y$ there exists
a continuous function $w: Y \into X^*$ such that
$q = w^{-1} \circ [*]$.
The $[*]$ is the canonical image-transformation
from $X$ to $X^*$ given by 
$A \mapsto [*] (A) \dlik A^* \dlik \{\sigma \in X^* \st \sigma (A) = 1 \}$.
The space $X^*$ is the set of simple quasi-measures $\sigma$
($\sigma (A)$ equals $0$ or $1$) equipped
with the weak topology: 
$\sigma \mapsto \sigma (a)$ is continuous for all bounded continuous $a$.
In terms of quasi-homomorphisms the factorization is 
as above with
$(w^{-1} a) (y) = a (w (y))$,
and $([*] b) (\sigma) = \sigma (b)$ is the quasi-linear Gelfand transformation.
This means in particular that $[*]$ is a quasi-homomorphism
from $C_b (X)$ to $C_b (X^*)$,
and $w^{-1}$ is a (quasi-)homomorphism from
$C_b (X^*)$ to $C_b (Y)$.
The function $w$ is unique and given by $w = q^* \circ \iota_Y$,
where $\iota_Y : Y \into Y^*$ is the inclusion which maps
$y$ to $\delta_y$,
and $q^*$ maps $Y^*$ into $X^*$.
The factorization result is summarized by the following
commutative diagrams:
\[
\xymatrix{
  {C_b (X^*)} \ar@{-->}[dr]^{w^{-1}} \ar[r]^{(q^*)^{-1}} 
& {C_b (Y^*)} \ar[d]^{\iota_Y^{-1}} \\
  {C_b (X)} \ar[u]^{[*]} \ar[r]_{q}        
& {C_b (Y)}
}\;\;\;\;\;\;\;\;\;\;\;\;\;\; \xymatrix{
  {{\cal A} (X^*)} \ar@{-->}[dr]^{w^{-1}} \ar[r]^{(q^*)^{-1}} 
& {{\cal A} (Y^*)} \ar[d]^{\iota_Y^{-1}} \\
  {{\cal A} (X)} \ar[u]^{[*]} \ar[r]_{q}        
& {{\cal A} (Y)}
}\;\;\;\;\;\;\;\;\;\;\;\;\;\; \xymatrix{
  {X^*} & {Y^*} \ar[l]_{q^*} \\
  {\;}   & Y \ar[u]_{\iota_Y} \ar[ul]_{w}}
\]
As an indication of another possible application of this theory
we quote Aarnes \cite[p.1]{AARNES:IMAGE}:
{\sl Once defined, image-transformations take on a life of their
own. In some sense they seem to be better vehicles for the litteral
transfering of an ``image'' or a message than ordinary functions,
since they allow for the possibility that images of ``small'' sets
will vanish, i.e. equal the empty set.}  

In the following we include 
normalization, compact-regularity and additivity
in the definitions of quasi-measures, quasi-integrals,
image-transformations, and quasi-homomorphisms in order to simplify
the language.
We follow the notational conventions:
$K,L,M$ are compact sets;
$F,G,H$ are closed sets;
$U,V,W$ are open sets;
$A,B,C$ are images;
${\cal A} = {\cal A} (X)$ is the set of images in a Hausdorff space $X$;
$k,l,m$ are real valued continuous functions with compact support;
$a,b,c$ are real valued bounded continuous functions;
and $C_b = C_b (X)$ is the set of real valued bounded continuous functions on $X$.

The first version of this work was a result of a seminar
based on \cite{AARNES:IMAGE} in the spring of 1995.
We acknowledge comments from 
Anden{\ae}s, Knudsen, Rustad, and Aarnes who participated in the seminar.
The results here corresponds to generalizations of the first part of \cite{AARNES:IMAGE}
and are approximately unchanged from the seminar,
but the organization of the proofs is different.
Rustad \cite{Rustad01unboundedQI} refers to an earlier version of this work.
Aarnes and Grubb \cite{AarnesGrubb04image} treat image transformations
in completely regular spaces and their results complements the results in the following.

\section{Integration and the Riesz Representation Theorem.}

The aim in this section is to give the ingredients in the
proof of the Riesz representation theorem.

\begin{theo}
Let $X$ be a locally compact normal space.
A one-one correspondence between quasi-measures $\mu$
and quasi-integrals $\rho$ on $C_b$ is given by
\[
\rho (a) = \int a (x) \mu (dx),\;\;
\mu (U) = \sup_{k \le U} \rho (k) .
\]
The simple quasi-measures corresponds to the simple quasi-integrals. 
\end{theo}

We start with the development of an integration theory
based on quasi-measures in a Hausdorff space.
Some properties of quasi-measures are
summarized by:
A quasi-measure is monotone:
$A \subset B \;\; \Rightarrow \;\; \mu (A) \le \mu (B) .$
It is continuous in the following sense:
$U_\lambda \uparrow U$ $\Rightarrow$
$\mu (U_\lambda) \uparrow \mu (U)$, and
$F_\lambda \downarrow F$
$\Rightarrow$
$\mu (F_\lambda) \downarrow \mu (F)$.
If $X$ is locally compact,
then
$\mu (U) = \sup \{\mu (V) \st \overline{V} = K \subset U\}$.
The proof of these statements are similar to the proof of
the corresponding statements for image-transformations.
The main difference between a measure and a quasi-measure is
that the latter is not defined on an algebra of sets:
The union and intersection of two images need not be an image.
In certain cases it turns out that quasi-measures may be identified 
with measures, and in particular

\begin{proposition}
A quasi-measure $\mu$ on $\RealN$ is the restriction of a unique
Borel measure $\nu$.
\end{proposition}
\begin{proof}
Put $F (t) \dlik \mu (-\infty, t]$.
$F$ is right continuous 
$F (t) = \lim_{t_i \downarrow t} \mu (-\infty, t_i]$
since $\mu$ is continuous.
Monotonity of $\mu$ ensures that $F$ is the distribution
function of a unique Borel measure $\nu$.
Any open set is the disjoint union of a countable family of open intervalls,
so the restriction claim follows from
\begin{eqnarray}
\nu (a,b) & = & \lim_{b_i \uparrow b} \nu (a, b_i] = 
 \lim_{b_i \uparrow b} F (b_i) - F (a) =
 \lim_{b_i \uparrow b} \mu(-\infty, b_i] - \mu (-\infty,a] \nonumber \\
 & = & \lim_{b_i \uparrow b} \mu(-\infty, b_i) - \mu (-\infty,a] =
 \mu(-\infty, b) - \mu (-\infty,a] = \mu (a,b). \nonumber
\end{eqnarray}
\end{proof}

This Proposition is a special case of a more general recent result
\cite[p.4]{SHAKM}:
Every Baire quasi-measure on a Tychonoff space with Lebesgue
covering dimension $\le 1$ is the restriction of a 
finitely additive Baire measure.

Let $\mu$ be a quasi-measure on $X$.
If $f: X \into Y$ is continuous,
then $f^{-1}: {\cal A} (Y) \into {\cal A} (X)$ is an 
image-transformation.
It follows in particular that $\mu \circ f^{-1}$ is
a quasi-measure on $Y$,
as will be proven in the next section.
The particular case $Y = \RealN$ together with the previous 
Proposition gives us integration:

\begin{defin}
Let $\mu_a$ be the extension of $\mu \circ a^{-1}$
to a Borel measure on $\RealN$.
The integral $\mu (a)$ of $a$ with respect to $\mu$
is 
\[
\mu (a) 
\dlik \int a (x) \; \mu (dx) 
\dlik \int t \; \mu_a (dt) 
\dlik \mu_a (\id) .
\]
\end{defin}

We remark that this definition is consistent with the conventional
for ordinary measures due to the change of variable formula.
Borel measures on the real line are uniquely given by their values
on open sets, so the family $\{\mu_a \}$ is a consistent family of measures:
\[
\mu_{\phi (a)} (U) = \mu ((\phi (a))^{-1}(U)) = \mu_a \circ \phi^{-1} (U).
\]
This gives that a quasi-measure on a Hausdorff space $X$ gives a
quasi-linear functional:
\[
\mu (\phi (a)) = \int t \; \mu_{\phi (a)} (dt)
= \int t \; \mu_a \circ \phi^{-1} (d t)
= \int \phi (s) \; \mu_a (ds) = \mu_a (\phi) .
\]

The following Staircase Lemma is fundamental.

\begin{lemma}
Let $a \le b$.
For each $\delta > 0$
we have the decomposition
$a = a_1 + \cdots + a_n$,
$b = b_1 + \cdots + b_n$,
$a_i \in A (a) \cap  A (a_i + b_i)$,
$b_i \in A (b) \cap  A (a_i + b_i)$,
and 
$a_i \le b_i + \delta / n$.
\end{lemma}
\begin{proof} \cite[p.54]{AARNES:QUASI}
Choose a constant $M$ such that
$\tilde{a} \dlik a + M$, 
$\tilde{b} \dlik b + M + \delta$ obeys
$0 \le \tilde{a} \le \tilde{b} - \delta$.
Choose $0 = \beta_0 < \cdots < \beta_n = \beta \dlik \norm{\tilde{b}}$, 
$\beta_{i + 1} - \beta_i < \delta$, and define
\[
\phi (x) \dlik \left\{
\begin{array}{ll}
0 & x \le 0\\
x & 0 \le x \le \beta\\
\beta & x \ge \beta
\end{array}
 \right. , \;\;\;\;\;\;
\phi_i (x) \dlik \left\{
\begin{array}{ll}
0 & x \le \beta_{i - 1}\\
x -  \beta_{i - 1}  &  \beta_{i - 1}  \le x \le \beta_i \\
\beta_i - \beta_{i - 1} & x \ge \beta_i
\end{array}
 \right. .
\]
With $a_i \dlik \phi_i (\tilde{a}) - M/n$, 
$b_i \dlik \phi_i (\tilde{b}) - (M + \delta)/n$,
and the observation $\phi = \sum_i \phi_i$, 
we conclude $a = \sum_i a_i$, $b = \sum_i b_i$,
and $a_i \in A (a), b_i \in A (b)$.
We prove $a_i, b_i \in A (a_i + b_i)$,
or equivalently  
$\tilde{a}_i \dlik  \phi_i (\tilde{a}), 
\tilde{b}_i \dlik  \phi_i (\tilde{b}) \in A (\tilde{a}_i + \tilde{b}_i)$.
From $\tilde{a} \le \tilde{b} - \delta$ 
we conclude $\tilde{b}_i (x) = \beta_i - \beta_{i - 1}$
when $\tilde{a}_i (x) > 0$.
This gives $\tilde{a}_i \cdot (\beta_i - \beta_{i - 1} - \tilde{b}_i) = 0$,
$\tilde{a}_i, \beta_i - \beta_{i - 1} - \tilde{b}_i 
\in A (\tilde{a}_i - \beta_i + \beta_{i - 1} + \tilde{b}_i)$, 
and finally
$\tilde{a}_i, \tilde{b}_i \in A (\tilde{a}_i + \tilde{b}_i)$.
\end{proof}

\begin{proposition}
If $\rho : C_b (X) \into \RealN$ is positive
and quasi-linear, then
$a \le b \imply \rho (a) \le \rho (b)$, and
$\abs{\rho (a) - \rho (b)} \le \rho (1_X) \norm{a - b}$.
\end{proposition}
\begin{proof}
The staircase Lemma gives
$\rho (a) = \sum_i \rho (a_i) 
\le \delta \rho (1_X) + \sum_i \rho (b_i) = \delta \rho (1_X) + \rho (b)$
from which we conclude $\rho (a) \le \rho (b)$.
$a \le b + \norm{a - b}$ 
gives $\rho [a] \le \rho [b] + \rho [1_X] \norm{a - b}$
and a switch of $b$ and $a$ gives
$\abs{\rho [a] - \rho [b]} \le \rho [1_X] \norm{a - b}$.
\end{proof}

It follows in particular that a quasi-integral is monotone and continuous.
In \cite{AARNES:QUASI} it is proven that the integral
with respect to a quasi-measure on a compact Hausdorff space
is a quasi-linear functional. 
The following Proposition is a generalization to the case of 
locally compact Hausdorff spaces.

\begin{proposition}
Let $\mu$ be a quasi-measure on a locally compact Hausdorff space.
The quasi-integral with respect to $\mu$ is a quasi-integral, 
and the change of variable  identity
\[
\mu (\phi (a)) = \mu_a (\phi)
\]
is valid for all continuous $\phi  : \RealN \into \RealN$.
The quasi-integral from a simple quasi-measure is a
simple quasi-integral.
\end{proposition}
\begin{proof} (i) The identity was proven above for a
general Hausdorff space.
Every element in $A (a)$ is on the form $\phi (a)$,
so linearity on $A (a)$ follows from the above identity and
linearity of $\mu_a$.
Positivity follows from the change of variable identity applied
to the function $\phi (t) = t^2$.
Finally $\mu (1) = \mu (1_\RealN (a)) = \mu_a (1_\RealN) = 1$.\\
(ii) If $a_\lambda \uparrow a$, then 
$\mu \circ a_\lambda^{-1} (t, \infty) \uparrow \mu \circ a^{-1} (t, \infty)$
from the continuity of $\mu$ and 
$a_\lambda^{-1} (t, \infty) \uparrow a^{-1} (t, \infty)$.
Helly's second theorem \cite[p.53]{LAMPERTI1} 
applied to the distribution functions 
$F (t) \dlik \mu \circ a^{-1}(-\infty, t]$ and 
$F_\lambda (t) \dlik \mu \circ  a_\lambda^{-1} (-\infty, t]$
gives $\mu (a_\lambda) \uparrow \mu (a)$.\\
(iii) Now we need local compactness.
The set $\Lambda \dlik \{k \st k \preceq X \}$ is 
directed by the conventional $\le$ for real-valued functions.
Given $a \in C_b (X)$ it follows that 
$a_k \dlik a \cdot k \uparrow a$, 
when $X$ is locally compact.
This, together with (ii), imply the 
regularity $\mu (a) = \sup_{k \le a} \mu (k)$.\\
(iv) If $\mu = \sigma$ is simple, then
\begin{equation}
\sigma_a = \delta_{\sigma (a)}
\end{equation}
since $\sigma_a = \delta_z$ is a simple Borel measure,
and $\sigma (a) = \sigma_a (\id) = z$.
Multiplicativity follows from 
$\sigma (\phi (a)) = \sigma_a (\phi) = \phi (\sigma (a))$ 
applied to the function $\phi = \phi_1 \phi_2$.
\end{proof}

In the above we proved quasi-linearity and 
normality ($a_\lambda \uparrow a \imply \mu(a_\lambda) \uparrow \mu (a)$)
in the case of a quasi-measure on a general Hausdorff space.
This monotone convergence theorem for nets holds
in a general Hausdorff space for a general quasi-integral $\rho$ due to
compact-regularity:
Let $k \preceq a$ with
$\rho [a]  \le \rho [k] + \epsilon$,
which is possible due to regularity.
Monotone convergence $a_\lambda \uparrow a$ gives
uniform convergence on $\supp k$ (Dini's Lemma),
and then a $\lambda_k$ such that $k \le a_\lambda + \epsilon$ whenever
$\lambda \ge \lambda_k$.
Quasi-linearity and monotonity give
$\rho [a] \le \rho [a_\lambda] + 2 \epsilon$  if
$\lambda \ge \lambda_k$,
which proves that $\rho$ is normal.

If $K \le a \le U$,
then
$
\mu (K) \le \mu \circ a^{-1} \{1\} =
$ $
\int_{\{1\}} t \;\; \mu_a (dt) \le
\mu (a) = 
$ $\int_{(0,1]} t \;\; \mu_a (dt) \le
\int_{(0,1]} 1 \;\; \mu_a (dt) = \mu \circ a^{-1} (0,1] \le
\mu (U)$.
If $X$ is locally compact, 
this gives
\begin{equation}
\mu (U) = \sup_{k \preceq U} \mu (k), 
\;\;\;\;\; \mu (K) = \inf_{K \le k} \mu (k).
\end{equation}
This gives in particular that the integrals corresponding to
different quasi-measures are different,
and that the quasi-measure is determined 
by its corresponding integral as stated in the Riesz representation theorem.
We will now sketch the proof of
the second part of the Riesz representation theorem.
Let a quasi-integral $\rho : C_b (X) \into \RealN$ be given.
Define $\mu (U) \dlik \sup_{k \preceq U} \rho (k)$.
It follows that $\mu$ is additive and
can be extended to $\cal A$ by $\mu (F) \dlik 1 - \mu (F^c)$.
The regularity gives  $\mu (U) = \sup_{a \le U} \rho (a)$,
and therefore $\mu (F) = \inf_{F \leq a} \rho (a)$,
which gives additivity on the closed sets from normality.
Monotonity of $\rho$ gives that $U \subset F$
implies $\mu (U) \le \mu (F)$.
Normality and Urysohn gives $F \le a \le U$
from which $\mu (F) \le \mu (U)$ follows.
The set function $\mu$ is therefore additive and monotone.
Regularity follows from consideration of
$k \le \supp k \le V \le K = \overline{V} \le U$,
so $\mu$ is a quasi-measure.
We prove $\rho (a) = \mu (a)$.
The representation theorem applied
to $C (\Spec a)$ gives a Borel measure $\nu_a$
determined by 
$\phi \mapsto \rho [\phi (a)] = \int \phi (t) \;\; \nu_a (dt)$.
The claim follows if we prove
$\nu_a (\alpha, \beta) = \mu_a (\alpha, \beta)$ for an
arbitrary intervall $(\alpha, \beta)$.
Let
\[
\phi_n (t) \dlik \left\{ 
\begin{array}{ll}
n (t - \alpha)		& \alpha \le t < \alpha + 1/n \\
1			& \alpha + 1/n \le t < \beta - 1/n \\
-n (t - \beta)		& \beta - 1/n \le t < \beta\\
0			& \mbox{otherwise}
\end{array}
\right. ,
\]
so $\phi_n \uparrow (\alpha, \beta)$ and 
$\phi_n (a) \leq U \dlik a^{-1} (\alpha, \beta)$.
The monotone convergence theorem gives 
$\nu_a (\alpha, \beta) = \lim \rho (\phi_n (a))$
and from 
$\mu_a (\alpha, \beta) = \sup_{k \le U} \rho (k)$
we conclude $\mu_a = \nu_a$.
Finally we prove that $\mu$ is simple if $\rho$ is
a simple quasi-integral.
From $\rho(\phi(a)) = \phi (\rho (a))$ we conclude
$\mu_a = \delta_{\rho (a)}$.
Let $K$ be compact with $\mu (K) > 0$.
From the regularity of $\mu$ we can find $k$ and $U$
with $K \preceq k \preceq U$ and $\mu (U) \le \mu (K) + \epsilon$.
From $0 < \mu (K) \le \mu (k^{-1} \{1\}) = 1 \le \mu (U) 
\le \mu (K) + \epsilon$,
we conclude $\mu (K) = 1$.
From this and regularity it follow that $\mu$ is simple.

\section{Image-Transformations and the Aarnes Factorization Theorem.}

\begin{proposition}
If $q$ is an image-transformation,
then
\[
(i)\; q (A^c) = q (A)^c;\;\;
(ii)\; A \subset B \;\; \Rightarrow \;\; q (A) \subset q (B);\;\;
(iii)\; A \cap B = \emptyset \;\; 
      \Rightarrow \;\; q(A) \cap q(B) = \emptyset;\;\;
\]
\[
(iv)\; q (U) = \bigcup_{K \subset U} q (K);\;\;
(v)\; U_\lambda \uparrow U  \;\; \Rightarrow \;\; q(U_\lambda) \uparrow q(U) .
\]
\end{proposition}
\begin{proof}
(i) $Y {=} q (A \cup A^c) 
{=} q (A) \uplus q(A^c)$.
(ii) and (iv) Consider $F \subset U$,
so $q (U) = q (F \uplus (U \cap F^c))
{=} q (F) \cup q (U \cap F^c) \supset q (F)$.
Since compacts are closed we find
$q (U) \supset \bigcup_{K \subset U} q (K)$.
Let $y \in q (U)$.
Since $\{y \}$ is compact,
the regularity gives a $K \subset U$ with
$y \in q (K)$ and (iv) follows.
$U \subset V$ and (iv) gives
$q(U) \subset q(V)$.
$U \subset F$ gives $q (F) = q (U \uplus (F \cap U^c))
{=} q (U) \cup q (F \cap U^c) \supset q (U)$,
and finally $F \subset G$ gives $q (F^c) \supset q (G^c)$ and
$q (F) \subset q (G)$ from (i).
(iii) Because of additivity it is sufficient to consider the case
$U \cap F = \emptyset$, but then
$q (U) {\subset} q(F^c) {=} q(F)^c$, 
so $q (U) \cap q (F) = \emptyset$.
(v) Monotonity gives $\cup_{\lambda} q(U_\lambda) {\subset} q(U)$,
so only $\supset$ remains.
Let $y \in q (U)$.
Regularity gives a compact $K \subset U = \cup_\lambda U_\lambda$ with
$y \in q (K)$,
but then 
$K \subset U_{\lambda_1} \cup \cdots \cup U_{\lambda_n} 
\subset U_{\hat{\lambda}}$.
We conclude $y \in q (U_{\hat{\lambda}})$ from monotonity.
\end{proof}

The image-transformations are arrows in a category
{\bf Image} with objects ${\cal A} (X)$.
The identity arrow $id : {\cal A} (X) \into {\cal A} (X)$
is an image-transformation and:

\begin{proposition}
The composition of two image-transformations is an image-transformation.
\end{proposition}
\begin{proof}
Let $q: {\cal A} (X) \into {\cal A} (Y)$ and 
$p: {\cal A} (Y) \into {\cal A} (Z)$ be image-transformations.
$p \circ q : {\cal A} (X) \into {\cal A} (Z)$ is well defined, and:
$p \circ q (X) = p (Y) = Z$;
$p \circ q (U) = p (q (U)^{open}) = p (q (U))^{open}$; and
$p \circ q (A \uplus B) =$ 
$p ( q (A) \uplus q (B)) = $ 
$p ( q (A)) \uplus p( q (B))$.
It remains to prove regularity.
Assume $K \subset p \circ q (U)$.
$q (U)$ is open, so the regularity of $p$ gives 
$L \subset q (U)$ such that $K \subset p (L)$.
The regularity of $q$ gives $M \subset U$
such that $L \subset q (M)$.
We conclude 
$K \subset p (L) \subset p (q (M))$ from the monotonity of $p$.
\end{proof}

Let $X^\wedge$ denote the set of functions 
$\sigma : C_b (X) \into \RealN$ with the property
$\sigma (\phi (a)) = \phi (\sigma (a))$.
In the proof of the quasi-multiplicativity of the quasi-integral
from a simple quasi-measure we proved that $X^* \subset X^\wedge$.
We identify $X$ with a subset of $X^*$ by the injection
$x \mapsto \delta_x$.
A more abstract characterization of $X$ as a subset
of $X^*$ is given by
\begin{proposition}
If $\mu \in X^*$ is subadditive on open sets;
$\mu (U \cup V) \le \mu (U) + \mu (V)$,
then $\mu = \delta_x$.
In particular $\RealN^* \simeq \RealN$. 
\end{proposition}
\begin{proof}
The family 
${\cal F}_\mu \dlik \{F \st \mu (F) = 1 \}$ is
closed under intersection from complementation
of the subadditivity on open sets.
The continuity of $\mu$ gives measure one to the set
$F \dlik \cap_{G \in {\cal F}_\mu } G$,
and in particular $F \neq \emptyset$.
Assume that $F$ contains two different points $x,y$.
The additivity contradicts $\mu \{x \} = \mu \{y \} = 1$,
so we can assume $\mu \{y\} = 0$.
From $\mu \{y\} = \inf_{y \in U} \mu (U)$ it follows
that there exists an open set $U \ni y$ with $\mu (U) = 0$.
Then $U^c \in {\cal F}_\mu$,
which contradicts $y \in F$.
We have proven $\mu \{x \} = 1$.
The claim $\mu = \delta_x$ follows.\\
Since $\mu \in \RealN^*$ is the restriction of a
Borel measure, it is subadditive on open sets.
\end{proof}
\begin{proposition}
The spectrum of $a$ is 
$\Spec{a}$
$ = \overline{ {a (X)}} = \overline{ {X^* (a)}} = X^\wedge (a).$
\end{proposition}
\begin{proof}
We define $\Spec{a} = \overline{a (X)}$ in agreement with
the more general $C^*$ definition. 
Assume $\sigma \in X^{\wedge}$.
We prove $\sigma (a) \in \Spec{a}$ by a contradiction argument.
Assume $\sigma [a] \not\in \Spec{a}$.
Urysohn's Lemma gives 
$\Spec{a} \preceq \phi \preceq \{\sigma (a) \}^c$,
and 
$0 = \phi (\sigma (a)) = \sigma (\phi (a)) = \sigma (1_X) = 1$ 
is a contradiction.
The inclusion $X \subset X^* \subset X^\wedge$ 
together with $\sigma (a) \in \Spec{a}$ gives
$\Spec{a} = \overline{a (X)} = \overline{X^* (a)} = \overline{X^\wedge (a)}$.
Below we identify $X^\wedge$ with a compact Hausdorff
space, and the continuity of $\sigma \mapsto \sigma (a)$ gives
that $\overline{X^\wedge (a)} = X^\wedge (a)$ is compact.
\end{proof}

Each $\Spec a$ is a compact Hausdorff space and
Tychonoff gives us the compact Hausdorff product space 
$\prod_{a \in C_b (X)} \Spec a$,
which is the family of real valued functions $\psi$ with 
$\psi (a) \in \Spec a$ and initial topology from
the functions $\psi \mapsto \psi (a)$.
This gives the inclusions
\begin{equation}
X \subset X^* \subset X^\wedge \subset \prod_{a \in C_b (X)} \Spec a,
\end{equation}
and corresponding relative topologies on
$X, X^*$, and $X^\wedge$. 
The original topology on $X$ equals the relative topology on $X$
if $X$ is a Tychonoff space,
and in particular if $X$ is a locally compact Hausdorff space.
$X^{\wedge}$ is compact since it is closed: 
$\sigma [\phi (a)] = \lim \sigma_\lambda [\phi (a)] 
= \phi (\lim \sigma_\lambda [a]) = \phi (\sigma [a])$.
$X^*$ is a Hausdorff space as a subset of a Hausdorff space.
It is to be expected that $X^*$ is locally compact when
$X$ is locally compact,
but this is an open question.

\begin{proposition}
Let $X$ be a locally compact Hausdorff space.
The canonical image-transformation
$[*]: {\cal A} (X) \into {\cal A} (X^*)$ given by
\[
[*] (A) \dlik A^* \dlik \{\sigma \in X^* \st \sigma (A) = 1 \}
\]
is an image-transformation.
\end{proposition}
\begin{proof}
Elements in $X^*$ are
quasi-integrals, so
$
U^* = \{\sigma \st \sigma (U) = 1 \} =
 \{\sigma \st \exists a \le U \;\; \sigma (a) > 0 \} =
\bigcup_{a \le U} \{\sigma \st  \sigma (a) > 0 \}
$
is open, and
$
F^* = \{\sigma \st \sigma (F) = 0 \}^c =
 \{\sigma \st \sigma (F^c) = 1 \}^c =
((F^c)^*)^c
$
is closed.
Together with $X^* = \{\sigma \st \sigma (X) = 1 \}$ we
have verified the first two axioms.
Additivity follows from
$
(F \uplus G)^* = \{\sigma \st 1 = \sigma (F) + \sigma (G) \} =
F^* \uplus G^* .
$
Put 
${\cal O}_c (U) \dlik \{V \st U \supset \overline{V} \mbox{ is compact} \}$. 
Compact-regularity of $\sigma$ and local compactness gives
$
U^* = \{\sigma \st \exists V \in {\cal O}_c (U) \;\; \sigma (V) = 1 \} =
\bigcup_{ V \in {\cal O}_c (U)  } V^* 
$,
since $[*]$ is monotone.
Let $K \subset U^*$.
Since ${\cal O}_c (U)$ is closed under unions,
the above gives $V \in {\cal O}_c (U)$ with $K \subset V^*$,
and $L = \overline{V}$ gives $K \subset L^*$,
so $[*]$ is regular.
\end{proof}

\begin{proposition}
Let $X$ and $Y$ be Hausdorff spaces.
If $f: Y \into X$ is continuous,
then $f^{-1}: {\cal A} (X) \into {\cal A} (Y)$ is an image-transformation.
\end{proposition}
\begin{proof}
Only the regularity needs a proof.
If $K \subset f^{-1} (U)$,
then $L = f (K) \subset U$,
and $K \subset f^{-1} (L)$.
\end{proof}

The special case $X = \RealN$ of the above Proposition was used 
when we defined integration.
It turns out that the two above examples of 
image-transformations covers all cases,
in the sense that all image-transformation
are on the form $w^{-1} \circ [*]$ for some
continuous $w: Y \into X^*$.
We need some other aspects of the theory in order to prove this.
Our main motivation for the study of image-transformations
is the following result.

\begin{proposition}
Let $q$ be an image-transformation 
from $X$ to $Y$.
Any quasi-measure $\mu$ on $Y$ is pulled back
to a quasi-measure $q^* \mu \dlik \mu \circ q$ on $X$.
The adjoint $q^*$ maps $Y^*$ into $X^*$.
The adjoint map is anti-multiplicative:
$(p \circ q)^* = q^* \circ p^*$.
\end{proposition}
\begin{proof}
$\nu \dlik q^* \mu$ is clearly additive:
$\nu (A \uplus B) = \mu (q (A) \uplus q (B)) = \nu (A) + \nu (B)$.
$K \subset U \imply 
\mu (q (K)) \le \mu (q (U))$, so regularity will follow from
$\mu ( q(U)) \le \sup_{K \subset U} \mu (q (K))$.
Let $L \subset q (U)$.
Since $\mu$ is regular we are left with the proof of
$\mu (L) \le \sup_{K \subset U} \mu (q (K))$.
$q$ is regular, so there is a compact $K \subset U$
with $L \subset q (K)$.
The monotonity of $\mu$ gives $\mu (L) \le \mu (q (K))$.\\
$q^* (Y^*) \subset X^*$ follows from
(i) $\sigma (q (X)) = \sigma (Y) = 1$ and 
(ii) $\sigma (q (A)) \in \{0, 1 \}$.
The final claim is
$
(p \circ q)^* \mu (A) = 
\mu \circ p \circ q (A) =
q^* (p^* \mu) (A)
$.
\end{proof}

Integration with respect to a quasi-measure 
gives a quasi-integral.
One may also integrate with respect to an 
image-transformation, and the result is a quasi-homomorphism.
If $a \mapsto r (a) (y)$ is positive and quasi-linear for each
$y$, then $a \mapsto r (a)$ is monotone and
$\norm{r (a) - r(b)} \leq \norm{r (1_X)} \; \norm{a - b}$.
A quasi-homomorphism is therefore monotone and continuous.
If $r (\phi (a)) = \phi (r (a))$ for all continuous
$\phi : \RealN \into \RealN$,
then clearly axiom (i) is satisfied.
The condition  $r (\phi (a)) = \phi (r (a))$ is also 
necessary, from uniform approximation of $\phi$ with polynomials.

\begin{proposition}
Let $q$ be an image-transformation from a locally compact
Hausdorff space $X$ to a Hausdorff space $Y$.
The integral $q (a)$ defined by
$q (a) (y) \dlik r_q (a) (y) \dlik (q^* \delta_y) (a)$ is a quasi-homomorphism
from $C_b (X)$ to $C_b (Y)$,
and
$q \circ a^{-1} = q (a)^{-1}$.
If $\mu$ is a quasi-measure on $Y$,
then $(q^* \mu) (a) = \mu (q(a))$.
If $q$ and $p$ are
composable image-transformations,
then $r_{q \circ p} = r_q \circ r_p$.
\end{proposition}
\begin{proof}
Recall that $\sigma_a = \delta_{\sigma (a)}$ holds for
a simple quasi-measure $\sigma$.
The equality $q \circ a^{-1} = q (a)^{-1}$  follows from
the equivalence of the following statements:
$q [a] (y) \in A$;
$\;\;1 =  \delta_{q [a] (y)} (A) = \delta_{[q^* \delta_y] (a)} (A) 
=  (q^* \delta_y)_a (A)  = q^* \delta_y \circ a^{-1} (A)$;
$\;\; y \in q (a^{-1} (A))$.
This proves continuity of $y \mapsto q (a) (y)$ from the case $A = U$.
The case $A = (-\infty, -\norm{a}) \cup (\norm{a}, \infty)$ gives
$q (a)^{-1} (A) = q \circ a^{-1} (A) = q (\emptyset) = \emptyset$,
so $\norm{q (a)} \le \norm{a}$.
So far we have proven that $q: C_b (X) \into C_b (Y)$ is well defined.
The property $q (\phi (a)) = \phi (q (a))$ follows from
$q (\phi (a)) (y) = 
(q^* \delta_y) (\phi (a)) = (q^* \delta_y)_a (\phi) = 
\delta_{(q^* \delta_y) (a)} (\phi) = \phi (q(a)(y))$,
which gives that $q: A(a) \into A(q(a))$ is a 
surjective algebra homomorphism.
We prove $q (a) (y) = \sup_{k \leq a} q(k) (y)$.
Quasi-linearity gives $\ge$.
Assume $q (a) (y) \in U$, 
or equivalently $y \in q (a^{-1} (U))$.
The claim follows if we can find $k \le a$ with 
$y \in q (k^{-1} (U))$.
The regularity of $q$ gives $K \subset a^{-1} (U)$ with
$y \in q (K)$.
Urysohn gives us $l$ with $K \le l \le X$.
With $k \dlik a l$, we conclude $k \le a$, and
$K \subset k^{-1} (k (K)) = k^{-1} (a (K)) \subset k^{-1} (U)$. 
This gives $y \in q (k^{-1} (U))$.
Let $\mu$ be a quasi-measure on $Y$.
Since $q (a)^{-1} = q \circ a^{-1}$, we
get $\mu_{q (a)} = (\mu \circ q)_a$,
and the change of variable formula
$\mu (q (a)) = (q^* \mu)(a)$ follows.
If $q,p$ are composable, then
$q (p (a)) (z) = $
$(q^* \delta_z)(p (a)) = $
$(\delta_z \circ q)_{p(a)} (id) = $
$(\delta_z \circ q \circ p)_{a} (id) = $
$[(q \circ p)^* \delta_z] (a)$.
\end{proof}

It should be observed that $a \mapsto q (a) (y)$ is
a simple quasi-integral for each $y$.
An alternative proof
of the regularity claim for $a \mapsto q (a)$
follows from
the corresponding statement for quasi-integrals. 
The special case $\mu = \delta_y$ in the change of variable
formula $\mu (q (a)) = (q^* \mu)(a)$ 
gives back the definition $q (a) (y) = (q^* \delta_y)(a)$.

\begin{theo}
Let $X$ be a locally compact Hausdorff space and
let $Y$ be a Hausdorff space.
There is 1-1 correspondence between
image-transformations $q: {\cal A} (X) \into {\cal A} (Y)$
and continuous functions $w : Y \into X^*$ given by
$q = w^{-1} \circ [*]$ and
$w = q^* \circ \iota_Y$,
with $\iota_Y (y) \dlik \delta_y$.
\end{theo}
\begin{proof}
We prove that $w \dlik q^* \circ \iota_Y$ is continuous
when $q$ is an image-transformation.
Fix $a \in C_b (X)$.
It is sufficient to prove continuity of
$y \mapsto w (y) (a) = q^* \circ \iota_Y (y) (a) =$
$(q^* \delta_y) (a) = q (a) (y)$,
but this follows since $q (a) \in C_b (Y)$.
Equality $q = w^{-1} \circ [*]$ follows from 
the following equivalent statements:
\[
y \in q (U);\;\;
1 = \delta_y (q (U));\;\;
w (y) = (q^* \circ \iota_Y)(y) = q^* \delta_y \in U^*;\;\;
y \in w^{-1} ([*](U)).
\]
Let  $w : Y \into X^*$ be continuous.
Since composition of image-transformations produce
image-transformations,
it follows that $q \dlik w^{-1} \circ [*]$ is an image-transformation.
Equality  $w = q^* \circ \iota_Y$ follows by inspection
of:
\[
1 = (q^* \circ \iota_Y (y))(A) = \delta_y (q (A));\;\;
y \in q (A) = w^{-1} (A^*);\;\;
w (y) (A) = 1.
\]
\end{proof}

Continuity of $w = q^* \circ \iota_Y$ follows also
from continuity of $q^*$ and $\iota_Y$,
which follows easily by consideration of convergent nets
in $Y^*$ respectively $Y$.
The proof of the remaining statements in the commutative diagrams 
in the introduction is now rather straightforward,
and left to the reader.

If $q$ is an image-transformation and
$K \leq a \leq U$,
then  $K \subset a^{-1} \{ 1\} \subset a^{-1} (0,\infty) \subset U$,
so
$q (K) \subset q (a)^{-1} \{1\} \subset q (a)^{-1} (0,\infty) \subset q (U)$.
When $X$ is locally compact, Urysohn gives us
\begin{equation}
q (U)
 = \bigcup_{K \subset U} q (K)
 =  \bigcup_{k \preceq U} q (k)^{-1} \{1\}
 =  \bigcup_{k \preceq U} q (k)^{-1} (0,\infty).
\end{equation}
The above indicates a 1-1 correspondence between
image-transformations and quasi-homomorphisms.
  
\begin{theo}
Let $X$ be a locally compact normal space.
If $r: C_b(X) \into C_b(Y)$ is a quasi-homomorphism, 
then there exists a unique image-transformation $q$
such that $r (a) = q (a)$.
\end{theo}

\begin{proof}
Put $q (U) \dlik \bigcup_{k \preceq U} r (k)^{-1} (0,\infty)$.
From $a_l \dlik a l \uparrow a $ with $l \preceq U$
and the monotone convergence theorem,
we conclude  $q (U) = \bigcup_{a \le U} r (a)^{-1} (0,\infty)$.\\
We prove additivity $q (U \uplus V) = q (U) \uplus q (V)$
on open sets:
$a \le U$ and $b \le V$ give
$a b = 0$, $a, b \in A (a - b)$, 
$0 = r (a b) = r (a) q (b)$,
$r (a)^{-1} (0,\infty) \cap r (b)^{-1} (0,\infty) = \emptyset$,
and finally $q (U) \cap q (V) = \emptyset$.
$\supset$ follows from monotonity on open sets.
We prove $\subset$: Let $y \in q (U \uplus V)$.
Then there exists $c \le U \cup V$ with $r (c) (y) > 0$.
It follows that $c = a + b$ with $a \le U$ and $b \le V$.
From $a,b \in A (a - b)$ it follows that
$r(a) (y) + r(b) (y) = r (c) (y) > 0$,
$r (a) (y) > 0$ or $r (b) (y) > 0$,
and finally $y \in q (U) \cup q (V)$.\\
$q (X) = Y$: Let $y \in Y$. 
The regularity gives 
$1 = r[1_X] (y) = \sup_{a \le X} r[a] (y)$ so
there exists  $a \le X$ such that $r [a] (y) > 0$.
Since $q$ is additive on open sets
we can extend $q$ to $\cal A$ by $q (F) = q (F^c)^c$.
It follows that $q (F) = \bigcap_{a \ge F} r (a)^{-1} \{1\}$.
This, and normality, proves additivity on closed sets.\\
It is clear that $q$ is monotone on open sets and
on closed sets.
Monotonity of $r$ gives that $U \subset F$
implies $q (U) \subset q (F)$.
Normality and Urysohn gives $F \le a \le U$
from which $q (F) \subset q (U)$ follows.
The set function $q$ is therefore additive and monotone.
Regularity follows from consideration of
$k \le \supp k \le V \le K = \overline{V} \le U$,
so $q$ is an image-transformation.\\
We prove $r (a) (y) = q (a) (y)$:
The map $a \mapsto r (a) (y) = \sigma (a)$ defines a
simple quasi-integral $\sigma$ with a corresponding
simple quasi-measure also denoted by $\sigma$.
Equality $\sigma = q^* \delta_y$ follows from equivalence of the
following statements:
\[
1 = \sigma (U) = \sup_{k \preceq U} \sigma (k);\;\;
\exists k \preceq U\; r(k)(y) > 0;\;\;
y \in q (U);\;\;
1 = \delta_y (q (U)).
\]
\end{proof}

\section{Examples.}

{\bf 4.1. The Aarnes Measure on the Square.}\\
The following example is abstracted from Aarnes \cite{AARNES:QUASI}.
A set $A$ in the unit square $X = [0,1]^2$ is solid if
$A$ and its complement $A^c$ are both connected.
Let $\partial X$ denote the border of $X$.
Put $\mu (A) = 1$ if $A$ contains the border or if
$A$ intersects both the border and $\{(1/2, 1/2)\}$,
and put $\mu (A) = 0$ otherwise.
This defines a 0-1 valued set function $\mu$ on the class of
solid sets in $X$.
If $F$ is closed and connected in $X$, 
then $F^c$ is a countable disjoint
union of open solid sets $U_1, U_2, \ldots$,
and we extend $\mu$ by $\mu (F) \dlik 1 - \mu (U_1) - \mu (U_2) - \cdots $. 
The set function is extended to the class ${\cal C}_0$ of
disjoint finite unions of
connected closed sets by 
$\mu (F_1 \uplus \cdots \uplus F_n) = \mu (F_1) + \cdots + \mu (F_n)$.
If $U$ is open, $\mu (U)$ is defined to be the supremum of
$\mu (K)$ for $K \subset U$, $K$ in  ${\cal C}_0$.
Finally $\mu (U^c) \dlik 1 - \mu (U)$ extends $\mu$ to
the class $\cal A$ of sets which are open or closed.
The set function $\mu$ is a proper quasi-measure since
it is not subadditive.
We refer to this quasi-measure as the Aarnes measure on the square.

\noindent{\bf 4.2. A Non-Linear Integral.}\\
Let $a$ be the ``pyramidal'' \cite[p.65]{AARNES:QUASI} 
function on $X$ whose graph
is given by the four planes which contain the point
$(1/2,1/2,1)$ in $\RealN^3$ and respectively the four sides
of $\partial X$.
In particular $a^{-1} (\{0\}) = \partial X$ which has
Aarnes measure $\mu (\partial X) = 1$.
This gives $\mu_a = \delta_0$, and $\mu (a) = 0$.\\
Let the graph of $b \ge 0$ be given by the plane
containing $F = \{(t,s,0) \st 0 \le t \le 1,\; 0 \le s \le 1/2 \}$,
and the plane containing the line segments 
$\{(t,1/2,0) \st 0 \le t \le 1 \}$
and
$\{(t,1,1) \st 0 \le t \le 1 \}$.
Since $b = 0$ on $F$, and $\mu (F) = 1$,
it follows that $\mu (b) = 1$. 
Let $G$ be the closed triangle in $X$ with 
corners at $(0,1),(1/2,1/2),(1,1)$.
The function $a + b$ equals $1$ on $G$, 
and $\mu (G) = 1$, 
so the quasi-integral $\mu$ is nonlinear 
\begin{equation}
1 = \mu (a + b) \neq \mu (a) + \mu (b) = 0 + 0 = 0 .
\end{equation}

\noindent{\bf 4.3. The 3-point Quasi-Measure and Quasi-Measures in the Plane.}\\
Consider again the unit square $X$.
A (normalized!) quasi-measure $\mu$ is parliamentary 
\cite{KNUDSEN:TOP} if $\mu (A) \neq 1/2$ for all $A$. 
An example is given by the ordinary
3 point measure $\mu_{p,q,r} = (\delta_p + \delta_q + \delta_r)/3$.
A simple quasi-measure $\tilde{\mu}$ is defined on
solid sets $K$ by
$\tilde{\mu} (K) = 1_{(1/2,1]} (\mu (K))$
and extended to $\cal A$ as in 4.1.
The 3-point quasi-measure is $\tilde{\mu}_{p,q,r}$.
It follows easily that the  3-point quasi-measure
is not a measure.
Let $\tilde{\mu}_{p,q,r}^n$ be the 3-point quasi-measure
on $X_n \dlik [-n,n]^2$, 
let $\iota_n$ be the natural injection of $X_n$ into the plane,
and conclude that 
$\mu_n \dlik \tilde{\mu}_{p,q,r}^n \circ {\iota_n}^{-1}$
is a quasi-measure on the plane.
Define 
$\mu_\infty (A) \dlik \lim \mu_n (A)$ for all images
$A$ in the plane.
The set function $\mu_\infty$ is additive,
but not compact-regular.
This example shows that
non-regular quasi-measures
arise quite naturally.
There is a Riesz representation theorem also for 
such measures  \cite{BOARDMAN:BAIRE}.  
Consider next 
a continuous $f_n: X \into \RealN^2$ with
$f_n (X) = [-1/4, 1/4]^2 + \{n, 0 \}$,
and put 
$\nu (A) \dlik \sum_{n \ge 1} 
\tilde{\mu}_{p,q,r} \circ {f_n}^{-1} (A) / 2^n $.
It follows that $\nu$ is a quasi-measure in
the plane which takes all values in $[0,1]$.
This example can be
generalized to the statement that
$A \mapsto \sum_n p_n \nu_n (A)$ is a quasi-measure
if $\sum_n p_n = 1$, $p_n \ge 0$, and each
$\nu_n$ is a quasi-measure.

\noindent{\bf 4.4. Image-Transformations.}\\
Let $\sigma$ be a simple quasi-measure on $X$.
An image-transformation $q$ \cite[p.10-11]{AARNES:IMAGE} from $X$ to $Y$
is defined by $q (A) = Y$ if $\sigma (A) = 1$ and
$q (A) = \emptyset$ if $\sigma (A) = 0$.
Let $Y$ be a subset of $X^*$,
let $w: Y \into X^*$ be the inclusion map,
and conclude that $q = w^{-1} \circ [*]$ is an
image-transformation.
Observe in particular that we may choose a finite $Y$.
More examples may be constructed and investigated by
continuous parametrizations $w$
of simple quasi-measures in $X^*$.

\section{Comments on Previous Results.}

In \cite{AARNES:QUASI} Aarnes establishes a Riesz representation
theorem for quasi-states in terms of quasi-measures on
compact Hausdorff spaces.
This has been generalized to
the case of a locally compact Hausdorff space $X$ in two different directions.
Aarnes \cite{AARNES:LOCALLY} arrives at quasi-measures which
are compact-regular, but not additive, by consideration
of the one point compactification of $X$.
Boardman \cite{BOARDMAN:BAIRE} obtains a representation theorem
for quasi-linear integrals on $C_b (X)$ in terms of 
quasi-measures which are additive, but not
compact-regular.
We introduced quasi-measures which are additive
and compact-regular.
Integration with respect to a quasi-measure $\mu$ is
defined as in \cite[p.46]{AARNES:QUASI}:
If $a: X \into \RealN$ is continuous, then
$\mu \circ a^{-1}$ is the restriction of a measure
$\mu_a$ on $\RealN$,
and $\mu (a) \dlik \int t \mu_a (dt)$ is the integral.
A novelty in our work is the simplified 
proof of the quasi-linearity of $a \mapsto \mu (a)$,
which avoids the consideration of the 
Riemann-Stieltjes integral \cite[p.46-52]{AARNES:QUASI}.
We remark also that this method of integration fails in
Boardman's more general case.
Our proof of the correspondence between image-transformations
and quasi-homomorphisms is different from the proof
by Aarnes.
Aarnes  \cite[p.13]{AARNES:QUASI} used the
fact that all homorphisms are on the form 
$a \mapsto a \circ w$ in the case of compact Hausdorff spaces,
but this is not available for locally compact spaces.
Differences like this are generically found on comparison
with \cite{AARNES:QUASI} which deals only with the compact case.

\newpage

{
\bibliography{image}
\bibliographystyle{plain}
}

\end{document}